\documentclass[11pt]{amsart}
\usepackage{geometry}                
\usepackage{graphicx}
\usepackage{amssymb}
\usepackage{amsmath}
\usepackage{epstopdf}
\usepackage{enumerate}

\usepackage{tikz}
\usepackage{circuitikz}

\usepackage[authoryear,sort&compress]{natbib}
\bibpunct{[}{]}{;}{n}{,}{,}

\usepackage[all]{xy}
\newtheorem{thm}{Theorem}

\newcommand{\pder}[2]{\frac{\partial #1}{\partial #2}}
\newcommand{\der}[2]{\frac{d #1}{d #2}}
\DeclareMathOperator*{\ext}{ext}
\DeclareMathOperator*{\cl}{cl}

\title{Constructing equivalence-preserving Dirac variational integrators with forces}
\author{Helen Parks}
\author{Melvin Leok}
\email{helen.parks@intel.com, mleok@math.ucsd.edu}
\address{Department of Mathematics, University of California at San Diego, La Jolla, CA 92093, USA}                                    

\begin{document}

\begin{abstract}
The dynamical motion of mechanical systems possesses underlying geometric structures, and preserving these structures in numerical integration improves the qualitative accuracy and reduces the long-time error of the simulation. For a single mechanical system, structure preservation can be achieved by adopting the variational integrator construction. This construction has been generalized to more complex systems involving forces or constraints as well as to the setting of Dirac mechanics. Variational integrators have recently been applied to interconnected systems~\cite{PaLe2017}, which are an important class of practically useful mechanical systems whose description in terms of Dirac structures and Dirac mechanical systems was elucidated in \cite{JaYo2014}. Since these interconnected systems are modeled as a collection of subsystems with forces of interconnection, we revisit some of the properties of forced variational integrators. In particular, we derive a class of Dirac variational integrators with forces that exhibit preservation properties that are critical when applying variational integrators to the discretization of interconnected Dirac systems. We close with a discussion of ongoing and future research based on these findings.
\end{abstract}

\maketitle

\section{Introduction}

Reframing analytic mechanics as geometric mechanics by considering mechanical systems as evolving on symplectic or Poisson manifolds reveals the depth of geometric structure in the physical world. Examining the range of geometric structures and their associated invariants---momenta, energy, symplectic forms, etc.---relevant to a particular mechanical motion provides a qualitative picture of the motion. Computational geometric mechanics allows these geometric structures to guide the development of numerical methods, leading to improved qualitative behavior of the numerical solutions.

Analytical mechanics and its geometric properties can be derived from variational or almost-variational principles. Variational integrators mimic this process by deriving integrators as the discrete evolution equations implied by a discrete variational principle. This process was first developed for discretizing Hamilton's principle in the case of conservative, nondegenerate, unconstrained Lagrangian systems. The discrete Hamilton's principle produces a set of discrete evolution equations known as the Discrete Euler--Lagrange (DEL) equations. Considered as a one-step map, the DEL equations are symplectic, approximately energy conserving, and satisfy a discrete Noether's theorem. This mirrors the symplecticity, energy conservation, and momentum conservation at the continuous level.

The variational construction has been extended to accommodate forcing, holonomic constraints \cite{MaWe2001}, and nonholonomic constraints \cite{CoMa2001}. Most recently, the variational construction was applied to Lagrange--Dirac mechanical systems, which is the Lagrangian view of Dirac mechanics \cite{LeOh2011}. Dirac mechanics generalizes both Lagrangian and Hamiltonian mechanics to a formulation that accommodates forces and constraints as well as degeneracy. Forces, constraints, and degeneracy each change the geometric structure of the system away from straightforward conservation laws, and Dirac geometry and mechanics elucidates the specifics of these changes \cite{YoMa2006a, YoMa2006b}. 

Dirac mechanics was also recently shown to be a useful setting for the study of interconnected systems \cite{JaYo2014}. Interconnected systems consist of multiple mechanical components joined together to operate as a whole. Many practical engineering systems can be conceptualized this way, and often it is more natural to model each component individually than to attempt to model the entire large-scale system monolithically. In \cite{PaLe2017}, we developed the foundations for incorporating the ease of subsystem-level modeling to geometric integration through an interconnected discrete Dirac simulation framework. That framework is based on incorporating the ideas of \cite{JaYo2014} into the discrete foundation developed in \cite{LeOh2011}.

This paper takes the first step by extending Dirac variational integrators to a formulation that includes discrete forcing. Connected subsystems exert interaction forces on one another, so accommodating forces may prove useful in interconnected applications. We can also view the interconnected modeling process as breaking a large (possibly conservative) system into smaller, more manageable parts. From this point of view, interaction forces are artificial modeling constructs, and we should ensure that they cancel properly in the final integration, so as to ensure that the large monolithic model is equivalent to the model based on interconnecting the component subsystems. This paper addresses that concern by examining notions of equivalence at both the continuous and discrete levels and providing a criterion for defining forced variational or Dirac integrators which are well defined with respect to changes in representation of equivalent systems.

The paper begins with a review of classic, forced, and Dirac variational integrators in Section 2. Section 3 examines order of accuracy and equivalence preservation in forced variational integrators, including several numerical examples. This study informs the later design and implementation of forced Dirac variational integrators. Section 4 describes the extension of Dirac variational integrators to forced Dirac variational integrators, closing with a numerical example. Section 5 summarizes the implications of this work and suggests future research directions.

\section{Background: A review of variational integrators}

\subsection{Classic variational integrators}

Consider a conservative, unconstrained, nondegenerate Lagrangian system. Given the Lagrangian $L:TQ\rightarrow \mathbb{R}$, we can construct a discrete Lagrangian $L_d:Q\times Q\rightarrow\mathbb{R}$, which can be viewed as approximating the Type I generating function for the symplectic time-$h$ flow map of the Euler--Lagrange vector field. With this, variational integrators are constructed from the discrete Hamilton's principle, which states that the discrete action sum is stationary,
\begin{equation*}
\delta \sum_{k=0}^{N-1} L_d(q_k,q_{k+1}) = 0,
\end{equation*}
for fixed endpoints, $\delta q_0 = \delta q_N = 0$.
Computing the infinitesimal variation of the discrete action sum, and collecting the terms involving $\delta q_k$ yields the discrete Euler--Lagrange equations,
\begin{equation*}
D_2L_d(q_{k-1},q_k) + D_1L_d(q_k,q_{k+1}) = 0.
\end{equation*}
Here, $D_1L_d$ and $D_2L_d$ denote the derivative of the discrete Lagrangian $L_d$ with respect its first and second argument, respectively. We view $Q\times Q$ as the discrete analog of the tangent bundle $TQ$ and the discrete Euler--Lagrange equations as defining a one-step map from $Q\times Q$ to itself that determines $(q_k,q_{k+1})$ from $(q_{k-1},q_k)$. 

One can also define a discrete flow along $T^*Q$ by introducing two discrete Legendre transforms, $\mathbb{F}^\pm L_d:Q\times Q\rightarrow T^*Q$,
\begin{align*}
\mathbb{F}^+L_d(q_k,q_{k+1}) &= (q_{k+1}, D_2L_d(q_k,q_{k+1})), \\
\mathbb{F}^-L_d(q_k,q_{k+1}) &= (q_k, -D_1L_d(q_k,q_{k+1})).
\end{align*}
These give two possible definitions for the momentum at the point $q_k$, which are defined in terms of the data $(q_k,q_{k+1})$ and $(q_{k-1},q_k)$,
\begin{align*}
p_{k,k+1}^- &= \mathbb{F}^-L_d(q_k,q_{k+1}) = -D_1L_d(q_k,q_{k+1}), \\
p_{k-1,k}^+ &= \mathbb{F}^+L_d(q_{k-1},q_{k}) = D_2L_d(q_{k-1},q_k).
\end{align*}
Then, the DEL equations enforce a momentum matching condition along the discrete flow, i.e.,
\begin{equation*}
p_{k-1,k}^+ = p_{k,k+1}^-.
\end{equation*}
Thus, the momentum $p_k = p_{k-1,k}^+ = p_{k,k+1}^-$ is well-defined along solutions of the DEL equations, and we can define a one-step map in phase space by
\begin{align*}
p_{k+1} &= D_2L_d(q_k,q_{k+1}), \\
p_k &= -D_1L_d(q_k,q_{k+1}).
\end{align*}
This map is called the discrete Hamiltonian map corresponding to the DEL equations.

The discrete Euler--Lagrange equations and discrete Hamiltonian map each preserve the appropriate symplectic form and satisfy a discrete Noether's theorem, so that symmetries in the discrete Lagrangian $L_d$ result in conservation of the component of the discrete momenta in the direction of the infinitesimal generators of the symmetry. As symplectic integrators, they also approximately conserve energy over exponentially long time scales. These conservation properties hold and ensure that the qualitative behavior of the Lagrangian system is well-approximated for any choice of $L_d$.

The choice of $L_d$ is guided by the existence of an exact discrete Lagrangian, $L_d^E$, which yields a discrete Hamiltonian map that samples the exact Hamiltonian flow. For points $q_0,q_1\in Q$, and $h$ sufficiently small, let $q_{01}(t)$ be the unique curve such that $q_{01}(0) = q_0, \  q_{01}(h) = q_1$, and $q_{01}(t)$ satisfies the Euler--Lagrange equations for $L$ on $[0,h]$. Then $L_d^E$ is given by
\begin{equation*}
L_d^E(q_0,q_1;h) = \int_0^h L(q_{01}(t),\dot{q}_{01}(t)) \ dt.
\end{equation*}
This is related to Jacobi's solution of the Hamilton--Jacobi equation, and can viewed as evaluating the action integral on the solution of an Euler--Lagrange boundary-value problem, or it can equivalently be described variationally,
\begin{equation*}
L_d^E(q_0, q_1;h)= \ext_{\substack{q\in C^2([0,h],Q) \\ q(0)=q_0, q(h)=q_1}} \int_0^h L(q(t), \dot q(t)) dt.
\end{equation*}
One can approximate the exact discrete Lagrangian $L_d^E$ by incorporating a wide variety of standard numerical techniques in order to obtain a computable discrete Lagrangian $L_d$ \cite{LeSh2011b, HaLe2012}. By the variational error analysis described in  \cite{MaWe2001}, it can be shown that if the discrete Lagrangian $L_d(q_0,q_1;h)=L_d^E(q_0,q_1;h)+\mathcal{O}(h^{r+1})$, then the associated discrete Hamiltonian map is an order $r$ approximation of the exact flow map.

\subsection{Forced variational integrators}

Given a continuous force $f_L:TQ\rightarrow T^*Q$ that is a fiber-preserving map over the identity, we introduce discrete forces $f_d^\pm:Q\times Q\rightarrow T^*Q$ that approximate the virtual work of the forces. Then, the forced discrete Euler--Lagrange equations are derived from the discrete Lagrange--d'Alembert principle,
\begin{equation}
\delta\sum_{k = 0}^{N-1}L_d(q_k, q_{k+1}) + \sum_{k=0}^{N-1}[f_d^-(q_k, q_{k+1})\cdot \delta q_k + f_d^+(q_k, q_{k+1})\cdot \delta q_{k+1}] = 0,
\label{forcedprinc}
\end{equation}
where as before, the endpoints are fixed, $\delta q_0 = \delta q_N = 0$. The forced discrete Euler--Lagrange equations are given by
\begin{equation*}
D_2L_d(q_{k-1},q_k)+f_d^+(q_{k-1},q_k)+D_1L_d(q_k,q_{k+1})+f_d^-(q_k,q_{k+1}) = 0.
\end{equation*}
The discrete Legendre transforms and discrete Hamiltonian map now incorporate the contribution of the discrete forces,
\begin{subequations}\label{discH_forced}
\begin{align}
p_{k+1} = \mathbb{F}^{f+}L_d(q_k,q_{k+1}) &= D_2L_d(q_k,q_{k+1})+f_d^+(q_k,q_{k+1}), \\
p_k = \mathbb{F}^{f-}L_d(q_k,q_{k+1}) &= -D_1L_d(q_k,q_{k+1})-f_d^-(q_k,q_{k+1}).
\end{align}
\end{subequations}
With these momentum definitions, the forced DEL equations can again be viewed as a momentum matching condition.

These equations reduce to the DEL equations in the absence of forcing, and they satisfy a forced discrete Noether's theorem for symmetries of the discrete Lagrangian where the discrete forces do no work in the direction of the  infinitesimal generators of the symmetry. Forced variational integrators exhibit better energy behavior than non-geometric integrators in practice, in the sense that the discrete energy evolution better reflects the exact energy evolution of the system. This lacks rigorous explanation since the forced equations are no longer symplectic nor energy preserving at either the continuous or the discrete level.

Once again, exact discrete quantities guide both the error analysis and the practical choice of $L_d$ and $f_d^\pm$. For points $q_0,q_1\in Q$, define $q_{01}(t)$ as above, except that we now require $q_{01}(t)$ to satisfy the \emph{forced} Euler--Lagrange equations,
\begin{equation*}
\pder{L}{q}(q,\dot{q}) - \der{}{t}\left(\pder{L}{\dot{q}}(q,\dot{q})\right) + f_L(q,\dot{q}) = 0,
\end{equation*}
for $L$ and $f_L$ on $[0,h]$. Then, $L_d^E$ and $f_d^{E\pm}$ are given by
\begin{align*}
L_d^E(q_0,q_1,h) &= \int_0^h L(q_{01}(t),\dot{q}_{01}(t)) \ dt,\\
f_d^{E+}(q_0,q_1,h) &= \int_0^h f_L(q_{01}(t),\dot{q}_{01}(t)) \cdot \frac{\partial q_{01}(t)}{\partial q_1} \ dt, \\
f_d^{E-}(q_0,q_1,h) &= \int_0^h f_L(q_{01}(t),\dot{q}_{01}(t)) \cdot \frac{\partial q_{01}(t)}{\partial q_0} \ dt.
\end{align*}
We obtain computable $L_d$ and $f_d^\pm$ by approximating the curve $q_{01}(t)$ that satisfies the forced Euler--Lagrange boundary-value problem and the integrals that arise in the definition of $L_d^E, f_d^{E\pm}$. Again, the orders of these approximations determine the order of accuracy of the discrete Hamiltonian map. This fact is stated without proof in \cite{MaWe2001}, and we provide an explicit proof in the following section.

\subsection{Dirac variational integrators}

Dirac variational integrators were developed in \cite{LeOh2011} and generalize the variational integrator construction to the case of Dirac mechanics. In \cite{LeOh2011}, the authors develop both the variational theory of discrete Dirac mechanics and explicit discrete analogues of Dirac structures. The authors arrive at two formulations, (+) and ($-$)-discrete Dirac mechanics, stemming from a choice of generating function used to define the discrete analogue of the symplectic flat map, $\Omega_{d\pm}^\flat$.

Discretization begins from a continuous problem described by a Lagrangian function $L:TQ\rightarrow\mathbb{R}$ and a continuous constraint distribution $\Delta_Q \subset TQ$. The constraint distribution is determined by its annihilator, which we write as $\Delta_Q^\circ = \text{span}\{\omega^a\}_{a=1}^m \subset T^*Q$. In each formulation of discrete Dirac mechanics, we construct a discrete Lagrangian $L_d$ and discrete annihilating one-forms $\omega_{d\pm}^a$ based on the continuous problem and a choice of retraction $\mathcal{R}:TQ\to Q$.

\subsubsection{(+)-discrete Dirac mechanics}
In (+)-discrete Dirac mechanics, we define the discrete one-forms as
\begin{equation*}
\omega_{d+}^a(q_k,q_{k+1}) = \omega^a(q_k,\mathcal{R}_{q_k}^{-1}(q_{k+1})).
\end{equation*}
These then define a discrete distribution as follows
\begin{equation*}
\Delta_Q^{d+} = \{(q_k,q_{k+1})\in Q\times Q \ | \ \omega_{d+}^a(q_k,q_{k+1}) = 0, \ a = 1,\dots,m\}.
\end{equation*}
The (+)-discrete Lagrange--d'Alembert--Pontryagin principle is
\begin{equation*}
\delta\sum_{k=0}^{N-1}[L_d(q_k,q_k^+)+p_{k+1}(q_{k+1}-q_k^+)] = 0,
\end{equation*}
with variations that vanish at the endpoints, $\delta q_0 = \delta q_N = 0$, and discrete constraints $(q_k,q_{k+1}) \in \Delta_Q^{d+}$. We also impose a constraint on the variations $\delta q_k \in \Delta_Q(q_k)$ after computing variations inside the sum. The variable $q_k^+$ serves as the discrete analog to the introduction of $v$ in the continuous principle. This process produces the (+)-discrete Lagrange--Dirac equations of motion,
\begin{align*}
0 &= \omega_{d+}^a(q_k, q_{k+1}), \\
q_{k+1} &= q_k^+ \\
p_{k+1} &= D_2L_d(q_k,q_k^+), \\ 
p_k  &= - D_1L_d(q_k,q_k^+) + \mu_a\omega^a(q_k),
\end{align*}
where $a = 1,\dots,m$, and the last equation uses the Einstein summation convention. These equations simplify to the DEL equations in the unconstrained case, and they recover the nonholonomic integrators of \cite{CoMa2001}. They can also be expressed in terms of discrete Dirac structures as
\begin{equation}
(X_d^k, \mathfrak{D}^+L_d(q_k,q_k^+))\in D_{\Delta_Q}^{d+}.
\label{pDiscreteDiracStructFormulation}
\end{equation}
where $X_d^k = ((q_k,p_k),(q_{k+1},p_{k+1}))$ is the discrete vector field, $\mathfrak{D}^+L_d$ is the (+)-discrete Dirac differential, and $D_{\Delta_Q}^{d+}$ is the (+)-discrete Dirac structure induced by the continuous distribution $\Delta_Q$. The discrete symplectic flat map, $\Omega_{d+}^\flat$, contributes to the definition of $D_{\Delta_Q}^{d+}$. We show \eqref{pDiscreteDiracStructFormulation} now to highlight its similarity to the continuous expression for constrained Dirac mechanics, $(X,\mathfrak{D}L)\in D_{\Delta_Q}$. We provide explicit descriptions of the discrete objects as needed in the development of forced Dirac integrators below.

\subsubsection{($-$)-discrete Dirac mechanics} This formulation defines the discrete one-forms as
\begin{equation*}
\omega_{d-}^a(q_{k},q_{k+1}) = \omega^a(q_{k+1},\mathcal{R}_{q_{k+1}}^{-1}(q_k))
\end{equation*}
and uses the ($-$)-discrete Lagrange-d'Alembert-Pontryagin principle
\begin{equation*}
\delta\sum_{k=0}^{N-1} [L_d(q_{k+1}^-,q_{k+1}) - p_k(q_k-q_{k+1}^-)] = 0 \ .
\end{equation*}
We impose the same constraints on the variations of the discrete principle as in the (+) case, using $\Delta_Q^{d-}$ defined from $\{\omega_{d-}^a\}_{a=1}^m$ instead of $\Delta_Q^{d+}$. The variable $q_k^-$ now plays the role of the discrete velocity. This yields the ($-$)-discrete Lagrange--Dirac equations
\begin{align*}
0&=\omega_{d-}^a(q_k, q_{k+1}),\\ 
q_{k} &= q_{k+1}^- \\
p_{k} &= -D_1L_d(q_{k+1}^-,q_{k+1}), \\
p_{k+1}&= D_2L_d(q_{k+1}^-,q_{k+1}) + \mu_a\omega^a(q_{k+1}).
\end{align*}
Equivalently,
\begin{equation*}
(X_d^k,\mathfrak{D}^-L_d(q_{k+1}^-,q_{k+1}))\in D_{\Delta_Q}^{d-}.
\end{equation*}
Again, we avoid defining each discrete object until necessary, and the interested reader is referred to \cite{LeOh2011} for the details.

\section{Forced Integrators revisited}

In \cite{MaWe2001}, the authors assert without proof the existence of a theorem relating the order of a forced variational integrator to the order of approximation of $L_d, f_d^\pm$ to $L_d^E, f_d^{E\pm}$. Here, we provide an explicit proof of the most practical direction of that theorem and a recipe for constructing forced integrators of known order. We then analyze the conditions under which this process yields a well-defined, equivalence-preserving integrator and discuss the implications of equivalence-preservation in interconnected applications. The section closes with several numerical examples illustrating both our construction process and the consequences of equivalence-preservation.

\subsection{Determining the order of a forced variational integrator}
Theorem \ref{orderthm} below explicitly proves that the orders of approximation of $L_d$ and $f_d^\pm$ to $L_d^E$ and $f_d^{E\pm}$ determine the order of accuracy of the discrete Hamiltonian map they define. This gives the most useful direction of the order theorem mentioned in \cite{MaWe2001}, since the order of $L_d$ and $f_d^\pm$ are easier to calculate than the order of the integrator.

We need a few preliminary definitions before we state the theorem. First, we explicitly define what we mean by the order of $L_d, f_d^{\pm}$ and $\mathbb{F}^{f\pm}L_d$, using the same definitions as in \cite{MaWe2001}. Thus, a given $L_d$ is of order $r$ if there exist constants $C_L>0$, $h_L>0$ and an open subset $U_L\subset TQ$ with compact closure such that
\begin{equation*}
||L_d(q(0),q(h),h) - L_d^E(q(0),q(h),h)|| \leq C_Lh^{r+1},
\end{equation*}
for all solutions $q(t)$ of the forced Euler--Lagrange equations with initial conditions satisfying $(q(0),\dot{q}(0)) \in U_L$, and for all $h\leq h_L$. Similarly, $f_d^\pm$ are of order $r$ if there exist constants $C_{f\pm}>0$, $h_{f\pm}>0$ and open subsets $U_{f\pm}\subset TQ$ with compact closure such that
\begin{equation*}
||f_d^\pm(q(0),q(h),h) - f_d^{E\pm}(q(0),q(h),h)|| \leq C_{f\pm} h_{f\pm}^{r+1},
\end{equation*}
for all solutions $q(t)$ of the forced Euler--Lagrange equations with initial conditions satisfying $(q(0),\dot{q}(0)) \in U_f^\pm$, and for all $h\leq h_f^\pm$. We will say that $L_d, f_d^\pm$ are \emph{simultaneously of order $r$} if there exists $U \subset U_L\cap U_f^+\cap U_f^-$ such that $U$ is a nontrivial open set with compact closure.

For the discrete Legendre transforms to be of order $r$, we require constants $C^\pm>0, h^\pm>0$ and an open set $U^\pm\subset TQ$ such that
\begin{equation*}
||\mathbb{F}^{f\pm}L_d(q(0),q(h),h)-\mathbb{F}^{f\pm}L_d^{E}(q(0),q(h),h)|| \leq C^\pm(h^\pm)^{r+1}
\end{equation*}
for all solutions $q(t)$ of the forced Euler--Lagrange equations with initial conditions satisfying $(q(0),\dot{q}(0)) \in U^\pm$, and for all $h\leq h^\pm$.

\begin{thm}
Consider a hyperregular Lagrangian $L$ with force $f_L$, corresponding Hamiltonian $H$, and corresponding Hamiltonian force $f_H$. If $L_d$ and $f_d^\pm$ are simultaneously of order $r$, then
\begin{enumerate}[a.]
\item the forced discrete Legendre transforms $\mathbb{F}^{f\pm}L_d$ are of order $r$.
\item the discrete Hamiltonian map is of order $r$.
\end{enumerate}
\label{orderthm}
\end{thm}

\begin{proof}
a. Assume that $L_d$ and $f_d^\pm$ are simultaneously of order $r$. Then $L_d$ of order $r$ implies existences of a function $e$ such that
\begin{equation}
L_d(q(0), q(h), h) = L_d^E(q(0),q(h),h) + h^{r+1}e(q(0),q(h),h)
\label{ld_order_r}
\end{equation}
and $||e(q(0),q(h),h)||\leq C_L$ on $U$. Similarly, $f_d^+$ of order $r$ implies the existence of a function $e^+$ such that
\begin{equation}
f_d^+(q(0), q(h), h) = f_d^{E+}(q(0),q(h),h) + h^{r+1}e^+(q(0),q(h),h)
\label{fdplus_order_r}
\end{equation}
and $||e^+(q(0),q(h),h)||\leq C_{f+}$ on $U$.

Taking derivatives of \eqref{ld_order_r} with respect to $q(h)$ gives
\begin{equation*}
D_2L_d(q(0),q(h),h) = D_2L_d^E(q(0), q(h),h) + h^{r+1} D_2e(q(0),q(h),h) \ .
\end{equation*}
We assume both $L_d$ and $L_d^E$ have continuous derivatives in $q(h)$, so, by the definition of $e$ in \eqref{ld_order_r}, $D_2e$ is continuous and bounded on the compact set $\cl (U)$. Combining this with \eqref{fdplus_order_r}, we have
\begin{multline*}
D_2L_d(q(0),q(h),h) + f_d^+(q(0), q(h), h)\\
 = D_2L_d^E(q(0), q(h),h) + f_d^{E+}(q(0),q(h),h) \\
+ h^{r+1}[D_2e(q(0),q(h),h) + e^+(q(0),q(h),h)] \ .
\end{multline*}
From the above arguments, $D_2e(q(0),q(h),h) + e^+(q(0),q(h),h)$ is bounded on the set $\cl (U)$. Thus, $\mathbb{F}^{f+}L_d$ is of order $r$. Taking derivatives of $L_d$ with respect to $q(0)$ and using a similar calculation with $f_d^-$ shows that $\mathbb{F}^{f-}L_d$ is of order $r$.
\\\\
\noindent
b. Let $F^f_{L_d}$ denote the integrator defined by the forced discrete Euler Lagrange equations and $\tilde{F}^f_{L_d}$ denote the discrete Hamiltonian integrator defined by the forced discrete Legendre transforms in \eqref{discH_forced}. Then, from the definitions of $F_{L_d}, \tilde{F}_{L_d}$, and $\mathbb{F}^{f\pm}L_d$, we have the  commutative diagram \eqref{three_triangles}, which is the forced equivalent of diagram (1.5.3) in \cite{MaWe2001}.
\begin{equation}
\xymatrix{
 & (q_0,q_1) \ar[ldd]_(0.59){\mathbb{F}^{f-}L_d} \ar[rr]^{F^f_{L_d}} \ar[rdd]_(0.405){\mathbb{F}^{f+}L_d} & & (q_1,q_2) \ar[ldd]^(0.405){\mathbb{F}^{f-}L_d} \ar[rdd]^(0.59){\mathbb{F}^{f+}L_d}& \\
 & & & & \\
(q_0,p_0)\ar[rr]_{\tilde{F}^f_{L_d}} &  & (q_1,p_1) \ar[rr]_{\tilde{F}^f_{L_d}} & & (q_2,p_2)
}
\label{three_triangles}
\end{equation}
Thus, we can express the forced discrete Hamiltonian map as $\tilde{F}_{L_d}^f = \mathbb{F}^{f+}L_d \circ (\mathbb{F}^{f-}L_d)^{-1}$, directly analogous to the unforced discrete Hamiltonian map, $\tilde{F}_{L_d} = \mathbb{F}^{+}L_d \circ (\mathbb{F}^{-}L_d)^{-1}$. As such, the proof of b. from a. given in \cite{MaWe2001} carries over exactly, and we choose not to reproduce it here.
\end{proof}

The next theorem provides a recipe for constructing forced triples $(L_d, f_d^\pm)$ of known order.

\begin{thm}
Consider a $p^{th}$ order accurate boundary-value solution method and a $q^{th}$ order accurate quadrature formula. Assume both the Lagrangian and the forcing function are Lipschitz continuous in both variables. Use the boundary-value solution method to obtain approximations $(q^{i},v^{i}) \approx (q(c_ih), v(c_ih))$, at the quadrature nodes $c_i$ of the solution $(q(t),v(t))$ of the forced Euler--Lagrange boundary-value problem. Then the associated discrete Lagrangian given by
\begin{equation*}
L_d(q_0,q_1;h) = h\sum_{i=0}^n b_iL(q^i,v^i) \ ,
\end{equation*}
and the discrete forces given by
\begin{align*}
f_d^+(q_0,q_1;h) &= h\sum_{i=0}^n b_if(q^i,v^i)\frac{\partial q^i}{\partial q_1},\\
f_d^-(q_0,q_1;h) &= h\sum_{i=0}^n b_if(q^i,v^i)\frac{\partial q^i}{\partial q_0},
\end{align*}
all have order of accuracy $\min(p+1,q)$.
\label{recipe}
\end{thm}

\begin{proof}
The result follows from a few straightforward calculations. We begin with $L_d$.
\begin{align*}
L_d^E(q_0,q_1,h) &= \int_0^h L(q_{01}(t),\dot{q}_{01}(t)) \ dt \\
&= \bigg[h\sum_{i=1}^m b_i L(q_{01}(c_ih),\dot{q}_{01}(c_i h))\bigg] + \mathcal{O}(h^{q+1}) \\
&= \bigg[h\sum_{i=1}^m b_i L(q^i + \mathcal{O}(h^{p+1}), v^i + \mathcal{O}(h^{p+1}))\bigg] + \mathcal{O}(h^{q+1}) \\
&= h\sum_{i=1}^m b_i L(q^i,v^i) + h\sum_{i=1}^m b_i \mathcal{O}(h^{p+1}) + \mathcal{O}(h^{q+1}) \\
&=  L_d(q_0,q_1,h) + \mathcal{O}(h^{p+2}) + \mathcal{O}(h^{q+1}) \ .
\end{align*}
For $f_d^+$,
\begin{align*}
f_d^{E+}(q_0,q_1,h) &= \int_0^h f(q_{01}(t),\dot{q}_{01}(t))\cdot\frac{\partial q_{01}(t)}{\partial q_1} \ dt \\
&= \bigg[h\sum_{i=1}^m b_i f(q_{01}(c_ih),v_{01}(c_ih)) \cdot \frac{\partial q_{01}(c_ih)}{\partial q_1} \bigg] + \mathcal{O}(h^{q+1}) \\
&=\bigg[h\sum_{i=1}^m b_i f(q^i + \mathcal{O}(h^{p+1}), v^i + \mathcal{O}(h^{p+1})) \cdot \bigg( \frac{\partial q^i}{\partial q_1} + \mathcal{O}(h^{p+1})\bigg)\bigg]+ \mathcal{O}(h^{q+1}) \\
%& \\
&=\bigg[h\sum_{i=1}^m b_i\bigg(f(q^i, v^i) + \mathcal{O}(h^{p+1})\bigg)\cdot\bigg(\frac{\partial q^i}{\partial q_1} + \mathcal{O}(h^{p+1})\bigg)\bigg] + \mathcal{O}(h^{q+1}) \\
&= h\sum_{i=1}^m b_i f(q^i,v^i)\cdot \frac{\partial q^i}{\partial q_1} + \mathcal{O}(h^{p+2}) + \mathcal{O}(h^{2p+3}) + \mathcal{O}(h^{q+1}) \\
&=  f_d^+(q_0,q_1,h) + \mathcal{O}(h^{p+2}) + \mathcal{O}(h^{q+1}).
\end{align*}
A similar calculation shows that
\begin{align*}
f_d^{E-}(q_0,q_1,h) &=  f_d^-(q_0,q_1,h) + \mathcal{O}(h^{p+2})+ \mathcal{O}(h^{q+1}) \ .
\end{align*}
\end{proof}

Taken together, Theorems \ref{orderthm} and \ref{recipe} show that we can construct forced variational integrators of known order from a choice of a quadrature rule and a boundary-value solution method. In fact, we could choose up to three different quadrature rules and boundary-value solution methods to define $L_d$, $f_d^+$, and $f_d^-$. The proofs above would still apply, and the resulting integrator would have order min($p_1+1,p_2+1,p_3+1, q_1,q_2,q_3)$ for $p_i$ the orders of the boundary-value solutions and $q_i$ the orders of the quadrature rules. However, this produces integrators with unpredictable results, as discussed in the next section.

\subsection{Notions of equivalence and equivalence-preservation}

A forced Lagrangian system has equations of motion
\begin{equation*}
\pder{L}{q}(q,\dot{q}) - \der{}{t}\left(\pder{L}{\dot{q}}(q,\dot{q})\right) + f_L(q,\dot{q}) = 0.
\end{equation*}
Since the equations of motion are defined by the combination of $L$ and $f_L$, it is possible for pairs $(L^1,f_L^1)$ and $(L^2,f_L^2)$ with $L^1\neq L^2$ and $f_L^1\neq f_L^2$ to define the same equations of motion. We refer to this as \emph{equivalence} of the pairs $(L^1,f_L^1)$ and $(L^2,f_L^2)$. Most numerical methods simulate mechanics by numerically integrating the differential equations of motion, so they produce the same numerical approximation whether the motion was originally described using $(L^1, f_L^1)$ or $(L^2,f_L^2)$.

Variational integrators simulate mechanics using the discrete equations of motion defined by a particular choice of $L_d$ and $f_d^\pm$. Thus, a variational integrator is defined by the choice of discretizations $L\mapsto L_d$ and $f_L \mapsto f_d^\pm$. It is therefore natural to ask which discretization schemes for the discrete Lagrangian and discrete forces when applied to equivalent representations of forced Euler--Lagrange systems lead to equivalent discrete equations of motion. This reflects whether or not the resulting variational integrator is well-defined on the equivalence class of representations of forced Euler--Lagrange systems.

More explicitly, given two equivalent representations $(L^1,f_L^1)$ and $(L^2,f_L^2)$ of a forced Euler--Lagrange system, the resulting solution trajectories are the same. Since a variational integrator for a forced Euler--Lagrange system aims to approximate that solution, it is desirable to consider well-defined variational integrators that produce the same approximation irrespective of which of the two equivalent representations $(L^1,f_L^1)$ or $(L^2,f_L^2)$ it is applied to.
%A variational integrator intends to approximate continuous motion, and equivalent pairs $(L^1,f^1)$ and $(L^2,f^2)$ define the same motion. Thus, a well-defined variational integrator should produce the same approximation whether starting from $(L^1,f^1)$ or $(L^2,f^2)$.
Well-definedness in this sense will be important in interconnected applications where we intentionally alter the forced representation of a system to view it as a collection of interacting subsystems. In particular, if we considered a system with many components in terms of the constituent free-body diagrams, it is essential when combining the free-body diagrams for the internal forces to cancel out in order to recover the original system.

We discuss the implications of equivalence at the continuous and discrete levels. Then, we provide a simple criterion for constructing well-defined forced variational integrators.

\subsubsection{Continuous equivalence}
Suppose we have two canonical Lagrangians, $L^i = \frac{1}{2}\dot{q}^TM^i\dot{q} - V^i(q)$. Then, equality of the equations of motion implies
\begin{equation*}
-\nabla V^1(q) + M^1\ddot{q} +f_L^1(q,\dot{q}) = -\nabla V^2(q) + M^2\ddot{q} +f_L^2(q,\dot{q}) \ ,
\end{equation*}
so that
\begin{equation*}
f_L^1(q,\dot{q}) - f_L^2(q,\dot{q}) = \nabla(V^2-V^1)(q) + (M^2-M^1)\ddot{q}.
\end{equation*}
Comparing variables on each side, we conclude that $M^2 = M^1$ and $f_L^1(q,v)-f_L^2(q,v) = \nabla V^1(q)-\nabla V^2(q)$.

More generally, assume two Lagrangians of the form $L^i(q,v) = K(q,v)-V^i(q) = \frac{1}{2}g(v,v)-V^i(q)$ for some metric tensor $g$. Then, continuous equivalence again implies $f_L^1(q,v)-f_L^2(q,v) = \nabla V^2(q)-\nabla V^1(q)$.

\subsubsection{Notions of discrete equivalence.}
In \cite{MaWe2001}, the authors define \emph{strongly equivalent} discrete Lagrangians to be those that generate equivalent discrete Hamiltonian maps and \emph{weakly equivalent} discrete Lagrangians to be those that generate equivalent discrete Lagrangian maps. That is to say that strongly equivalent discrete Lagrangians will generate the same discrete sequence $\{(q_k, p_k)\}_{k=0}^N\subset T^*Q$, whereas weakly equivalent discrete Lagrangians will only generate the same sequence $\{q_k\}_{k=0}^N\subset Q$.

We can define strong and weak equivalence of discrete triples $(L_d^i, f_d^{i\pm})$ in the same way. From diagram \eqref{three_triangles}, we see that we have strong equivalence if and only if we have equivalence of the forced discrete Legendre transforms, $\mathbb{F}^{f\pm}L_d$, just as in the unforced case. Then, Theorem~\ref{orderthm} still holds for triples which are strongly equivalent to a triple meeting the assumptions. Moreover, Dirac variational integrators directly generalize the discrete Hamiltonian implementation of forced variational integrators. Thus, we focus on strong equivalence.

Strongly equivalent discrete triples $(L_d^1, f_d^{1\pm})$ and $(L_d^2, f_d^{2\pm})$ have equivalent forced discrete Legendre transforms, i.e. $\mathbb{F}^{f\pm}L_d^1 = \mathbb{F}^{f\pm}L_d^2$.  This implies that
\begin{equation*}
f_d^{2-}(q_k,q_{k+1}) - f_d^{1-}(q_k,q_{k+1}) = D_1L_d^1(q_k,q_{k+1})-D_1L_d^2(q_k,q_{k+1}),
\end{equation*}
and
\begin{equation*}
f_d^{1+}(q_k,q_{k+1}) - f_d^{2+}(q_k,q_{k+1}) = D_2L_d^2(q_k,q_{k+1}) - D_2L_d^1(q_k,q_{k+1}).
\end{equation*}
So each difference in discrete forcing in two strongly equivalent triples must be integrable in at least one of the position variables. If we think of $f_d^\pm$ as right and left discrete forces, then each discrete force difference must be integrable in its base-point variable, mirroring the continuous conclusion that $f_L^1(q,v)-f_L^2(q,v) = \nabla V^2(q)-\nabla V^1(q)$.

\subsubsection{Preserving continuous equivalence}
Theorems \ref{orderthm} and \ref{recipe} provide a means of constructing forced variational integrators of known order by choosing quadrature rules and boundary-value solution methods. It is tempting to try to optimize the overall discretization by tailoring the discretization of $L_d$, $f_d^+$, and $f_d^-$ individually, but the resulting integrator is no longer well-defined with respect to the equivalence relation defined above. The simplest way to generate an integrator that preserves continuous equivalence by our method is to choose the same quadrature rule and boundary-value solution method for all three discrete quantities.

In this case, we have
\begin{equation*}
L_d(q_0,q_1;h) = h\sum_{i=0}^n b_iL(q^i,v^i) \ ,
\end{equation*}
and
\begin{align*}
f_d^+(q_0,q_1;h) &= h\sum_{i=0}^n b_i f_L(q^i,v^i)\frac{\partial q^i}{\partial q_1}\\
f_d^-(q_0,q_1;h) &= h\sum_{i=0}^n b_i f_L(q^i,v^i)\frac{\partial q^i}{\partial q_0}.
\end{align*}
Suppose $(L^1,f_L^1)$ and $(L^2,f_L^2)$ are equivalent, so
\[f_L^1(q,\dot{q})-f_L^2(q,\dot{q}) = \nabla V^1(q) - \nabla V^2(q).\]
Constructing $(L_d^1, f_d^{1\pm})$ and $(L_d^2,f_d^{2\pm})$ from $(L^1, f_L^1)$ and $(L^2,f_L^2)$, we then have
\begin{align*}
f_d^{1+}(q_0,q_1;h) - f_d^{2+}(q_0,q_1;h) &= h\sum_{i=0}^n b_i (f_L^1(q^i,v^i)-f_L^2(q^i,v^i))\pder{q^i}{q_1} \\
&= h\sum_{i=0}^n b_i (\nabla V^2(q^i)-\nabla V^1(q^i))\pder{q^i}{q_1} \\
&= D_2L_d^2(q_0,q_1;h) - D_2L_d^1(q_0,q_1;h).
\end{align*}
Similarly, $f_d^{2-}(q_0,q_1;h) - f_d^{1-}(q_0,q_1;h) = D_1L_d^1(q_0,q_1;h) - D_1L_d^2(q_0,q_1;h)$. Thus, $(L_d^1, f_d^{1\pm})$ and $(L_d^2,f_d^{2\pm})$ are strongly equivalent.

\subsection{Numerical examples}
We construct our implementations by choosing a quadrature rule and a one-step map, using shooting to solve the boundary-value problem, which is analogous to the shooting-based variational integrator proposed in \cite{LeSh2011b}. Quadrature of order $q$ and a one-step map of order $p$ produce an integrator of order min($p,q$) since the shooting solution is order $p-1$.

Table \ref{OrderTable} shows this construction converges as predicted. The second-order integrator uses trapezoidal quadrature and a second-order Taylor's method shooting solution. The fourth-order integrator uses Simpson's Rule quadrature and a fourth-order Taylor's method shooting solution. Both were run on a damped harmonic oscillator with mass $m=1$, spring constant $k=1$, and damping coefficient $c=0.01$. The table shows the effect of doubling the number of time steps per period $T$ on the error after five periods. In both cases, the error decreases at the predicted rate.

\begin{table}
\begin{tabular}{| r | c | c | c | c |}
\hline
Steps per period & 20 & 40 & 80 & 160 \\
\hline
$2^{nd}$ order error & 0.2275 & 0.0557 & 0.0138 & 0.0035 \\
 at $t=5T$ & & & & \\
\hline
error ratio & -- & 4.0844 & 4.0362 & 3.9429 \\
\hline
$4^{th}$ order error & 0.6551$\times 10^{-4}$ & 0.0516$\times 10^{-4}$ & 0.0034$\times 10^{-4}$ & 0.0002$\times 10^{-4}$ \\
 at $t=5T$ & & & & \\
\hline
error ratio & -- & 12.696 & 15.176 & 17  \\
\hline
\end{tabular}
\caption[Fourth and second-order forced integrators show the expected convergence rates.]{Fourth and second-order forced integrators show the expected convergence rates. The second-order integrator uses Trapezoidal quadrature and second-order Taylor shooting. The fourth-order integrator uses Simpson's Rule quadrature and fourth-order Taylor shooting. Both methods simulate a damped harmonic oscillator with $m = k = 1$ and damping coefficient $c = 0.01$.}
\label{OrderTable}
\end{table}

Figure \ref{SecondOrderWD} illustrates the unpredictability of the results of non-equivalence-preserving discretizations. Two implementations use second-order Taylor's method shooting as the boundary-value solution method for both $L_d$ and $f_d^\pm$ with a Trapezoidal and/or Midpoint quadrature. The equivalence-preserving discretization uses Trapezoidal quadrature on both $L_d$ and $f_d^\pm$. The non-equivalence-preserving implementation uses a Trapezoidal quadrature on $L_d$ and Midpoint quadrature on $f_d^\pm$. Both integrators were run on a conservative harmonic oscillator eleven times, with the potential force increasingly represented as an external force, i.e.
\begin{equation*}
L(q,v) = \frac{1}{2}v^TMv - (1-\alpha)V(q) \ , \ f(q,v) = -\alpha \nabla V(q)
\end{equation*}
for $\alpha = 0,0.1,0.2\dots,1$. The step size was $h=0.05$ for all runs. All values of $\alpha$ produce the same continuous equations of motion. The equivalence-preserving discretization produces the same solution regardless of $\alpha$. The non-equivalence-preserving discretization produces $\alpha$-dependent solutions which veer away from the unforced representation's solution as $\alpha$ increase. We also compared a Midpoint-Midpoint construction to a Midpoint-Trazezoidal construction in the same way. The results lie directly atop the Trap-Trap vs. Trap-Mid results shown. This unpredictability of results makes non-equivalence-preserving integrators unsuited for applications like interconnected systems in which a system's forced representation will purposely be modified by modeling the component subsystems individually, with internal forces to account for the interaction with other subsystems.

\begin{figure}
\includegraphics[scale=0.6]{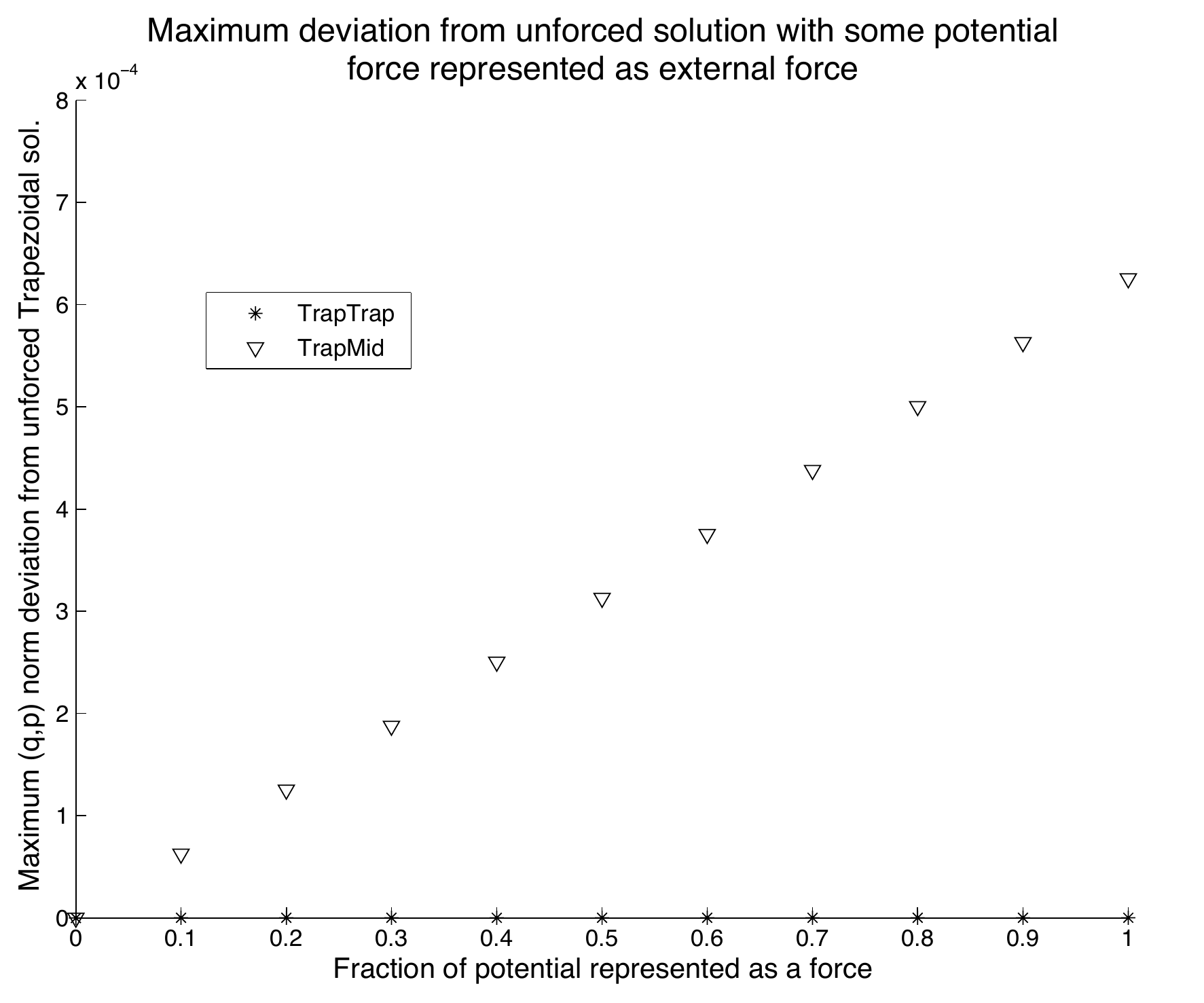}
\caption[Unpredictability of the results of non-equivalence-preserving discretizations.]{As the system's potential is increasingly represented as an external force, the non-equivalence-preserving construction produces solutions that veer further and further from the solution of an unforced representation. The equivalence-preserving construction produces the same solution whether the potential force is represented entirely through a potential within the Lagrangian or entirely as an external force.}
\label{SecondOrderWD}
\end{figure}

Figure \ref{MassivelyAwry} shows a different example where failure to preserve equivalence produces wildly unpredictable results. We constructed a shooting based variational integrator with fourth-order Runge--Kutta as the underlying one-step method, and used the trapezoidal rule on $L_d$ and $f_d^\pm$ to construct a second-order equivalence-preserving integrator, and Simpson's rule on $L_d$ and trapezoidal rule on $f_d^\pm$ to construct a second-order non-equivalence-preserving integrator. Both of these methods were applied to simulate a system whose equations of motion simplify to an unforced harmonic oscillator. We inserted an artificial potential of $100q^5$, then canceled that potential force with a forcing function. As can be seen from the figure, the non-equivalence-preserving formulation fails to capture the motion and exhibits a qualitatively different behavior. The example is artificial, but it highlights the need to preserve equivalence in developing a well-defined integrator.

\begin{figure}
\includegraphics[scale=0.6]{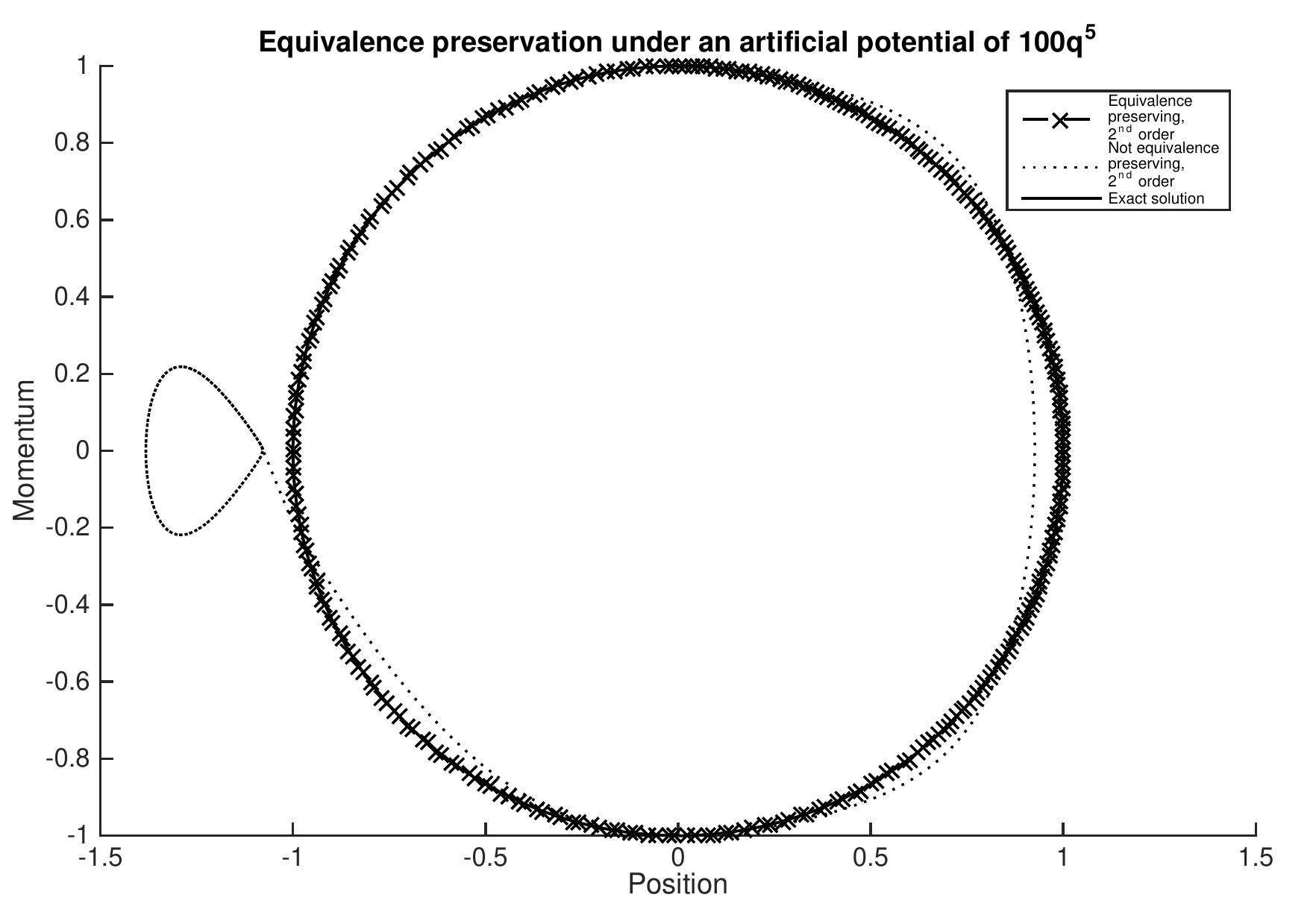}
\caption[Another example where failure to preserve equivalence produces wildly unpredictable results.]{Two integrators simulate a system whose equations of motion simplify to an undamped harmonic oscillator. An artificial potential of 100$q^5$ was added to the Lagrangian, with the artificial potential force canceled by an external force. The $2^{nd}$ order equivalence-preserving discretization is indistinguishable from the exact solution, while the $2^{nd}$ order non-equivalence-preserving discretization veers wildly away.}
\label{MassivelyAwry}
\end{figure}

\section{Dirac variational integrators with forces}
We now lay out the generalization of Dirac variational integrators to include forces. We begin from the variational perspective, where we can directly apply the ideas of forced variational integrators. We then formulate the resulting integrators from the perspective of discrete Dirac structures and finish with a basic numerical implementation.

\subsection{Variational formulation}

\subsubsection{(+)-discrete Dirac mechanics}
In \cite{LeOh2011}, the discrete variational principle for a Dirac system without external forces is the (+)-discrete Lagrange-d'Alembert-Pontryagin principle, given by
\begin{equation*}
\delta \sum_{k=0}^{N-1}[L_d(q_k,q_k^+) + p_{k+1}(q_{k+1} - q_k^+)] = 0,
\end{equation*}
where the variations vanish at the endpoints and are constrained to lie on a constraint distribution, and pairs of consecutive points in the discrete solution lie in a discrete constraint distribution.
We define the forced principle in direct analogy to \eqref{forcedprinc} as
\begin{equation}
\delta \sum_{k=0}^{N-1}[L_d(q_k,q_k^+) + p_{k+1}(q_{k+1} - q_k^+)] + \sum_{k=0}^{N-1}[f_d^-(q_k, q_{k}^+)\cdot \delta q_k + f_d^+(q_k, q_{k}^+)\cdot \delta q_{k}^+] = 0,
\label{pDiracForcedPrinc}
\end{equation}
where the second sum represents the virtual work associated with the external forces.
We have chosen the arguments of $f_d^\pm$ to match those of $L_d$, and we constrain the variations in the same way as \cite{LeOh2011}. Thus, $\delta p_k$ and $\delta q_k^+$ are arbitrary and $\delta q_0 = \delta q_n = 0$. We impose $\delta q_k\in \Delta_Q(q_k)$ on the remaining $\delta q_k$ after taking variations inside the sum, and we insist that $(q_k, q_{k+1})\in \Delta_Q^{d+}$. For the forced principle \eqref{pDiracForcedPrinc}, these constraints yield
%\begin{subequations}
\begin{align*}
0&=\omega_{d+}^a(q_k, q_{k+1}), \\ 
q_{k+1} &= q_k^+,\\
p_{k+1} &= D_2L_d(q_k,q_k^+) + f_d^+(q_k, q_k^+), \\
p_k&= - D_1L_d(q_k,q_k^+) - f_d^-(q_k, q_k^+) + \mu_a\omega^a(q_k).
\end{align*}
%\end{subequations}
These equations recover the (+)-discrete Lagrange--Dirac equations of \cite{LeOh2011} when $f_d^\pm = 0$ and the forced discrete Euler--Lagrange equations \cite{MaWe2001} in the absence of constraints. When both the forces and constraints are zero, we recover the classic discrete Euler--Lagrange equations.

\subsubsection{($-$)-discrete Dirac mechanics}
In \cite{LeOh2011}, ($-$)-discrete Dirac mechanics derives from the ($-$)-discrete Lagrange-d'Alembert principle,
\begin{equation*}
\delta \sum_{k=0}^{N-1}[L_d(q_{k+1}^-,q_{k+1})-p_k(q_k-q_{k+1}^-)] = 0,
\end{equation*}
with constraints on the variations, and the pairs of sequential points as before.
We add forces to this principle as
\begin{equation}
\begin{aligned}
\delta \sum_{k=0}^{N-1}&[L_d(q_{k+1}^-,q_{k+1})-p_k(q_k-q_{k+1}^-)]\\
&+ \sum_{k=0}^{N-1}[f_d^-(q_{k+1}^-, q_{k+1})\cdot \delta q_{k+1}^- + f_d^+(q_{k+1}^-, q_{k+1})\cdot \delta q_{k+1}] = 0,
\label{mforcedprinc}
\end{aligned}
\end{equation}
again choosing the arguments of $f_d^\pm$ to match those of $L_d$ and constraining the variations as in \cite{LeOh2011}. In this case, then, $\delta p_k$ and $\delta q_k^-$ are arbitrary and $\delta q_0 = \delta q_N = 0$. We impose $\delta q_k \in \Delta_Q(q_k)$ on the remaining $\delta q_k$ after taking variations inside the sums, and we insist that $(q_k, q_{k+1})\in \Delta_Q^{d-}$. With these conditions, computing variations of the principle \eqref{mforcedprinc} yields
%\begin{subequations}
\begin{align*}
0&=\omega_{d-}^a(q_k, q_{k+1}),\\
q_k &= q_{k+1}^-, \\
p_k &= -D_1L_d(q_{k+1}^-,q_{k+1}) - f_d^-(q_{k+1}^-,q_{k+1}), \\
p_{k+1}&= D_2L_d(q_{k+1}^-,q_{k+1}) + f_d^+(q_{k+1}^-, q_{k+1}) + \mu_a\omega^a(q_{k+1}).
\end{align*}
%\end{subequations}
In the unforced case, these equations recover the ($-$)-discrete Dirac equations of \cite{LeOh2011}. In the unconstrained case, we recover the forced discrete Euler--Lagrange equations of \cite{MaWe2001}. In the absence of both we again recover the classic discrete Euler--Lagrange equations.

\subsection{Discrete Dirac structure formulation}
The dynamics of an unforced, continuous Lagrange--Dirac dynamical system can be expressed in terms of Dirac structures as $(X, \mathfrak{D}L) \in D$, where $X$ is the partial vector field of the motion, $\mathfrak{D}L$ the Dirac differential of the Lagrangian $L$ and $D$ a Dirac structure on $T^*Q$. To accomodate forces in the continuous case, the force $F : TQ \to T^*Q$ is lifted to a map $\tilde{F} : TQ \to T^*T^*Q$ such that
\begin{equation*}
\langle \tilde{F}(q,v),w\rangle = \langle F(q,v), T\pi_Q(w) \rangle \ .
\end{equation*}
Locally, $\tilde{F}$ is given by $\tilde{F}(q,v) = (q,p, F(q,v),0)$. The forced Lagrange--Dirac dynamical system is then given by
\begin{equation}
(X, \mathfrak{D}L(q,v) - \tilde{F}(q,v)) \in D_{\Delta_Q}(q,p) \ .
\label{ctsForcedDiracStruct}
\end{equation}

Without forces, we can express (+) and ($-$)-discrete Dirac mechanics in terms of discrete Dirac structures as 
\begin{equation*}
(X_d^k, \mathfrak{D}^+L_d(q_k, q_k^+))\in D_{\Delta Q}^{d+},
\end{equation*}
and
\begin{equation*}
(X_d^k, \mathfrak{D}^-L_d(q_{k+1}^-, q_{k+1}))\in D_{\Delta Q}^{d-},
\end{equation*}
which amount to the conditions
%\begin{subequations}
\begin{align*}
(q_k,D_2L_d,-D_1L_d,q_k^+)& - (q_k,p_{k+1},p_k,q_{k+1})\in \Delta_{Q\times Q^*}^\circ , \\
(-D_1L_d, q_{k+1}, -q_{k+1}^-, -D_2L_d)& - (p_k, q_{k+1}, -q_k, -p_{k+1}) \in \Delta_{Q^*\times Q}^\circ .
\end{align*}
%\end{subequations}

To introduce forces into these expressions in a manner that is consistent with the discrete equations of motion that are derived variationally, we require
\begin{subequations}\label{ForceConds}
\begin{align}
(q_k,D_2L_d+f_d^+,-D_1L_d-f_d^-,q_k^+)& - (q_k,p_{k+1},p_k,q_{k+1})\in \Delta_{Q\times Q^*}^\circ , \\
(-D_1L_d-f_d^-, q_{k+1}, -q_{k+1}^-, -D_2L_d-f_d^+)& - (p_k, q_{k+1}, -q_k, -p_{k+1}) \in \Delta_{Q^*\times Q}^\circ .
\end{align}
\end{subequations}
To this end, define $\tilde{F}_{d+}$ and $\tilde{F}_{d-}$ as
\begin{align*}
\tilde{F}_{d+}(q_k, q_k^+) &= (0, -f_d^+(q_k, q_k^+), f_d^-(q_k, q_k^+), 0), \\
\tilde{F}_{d-}(q_{k+1}^-,q_{k+1}) &= (f_d^-(q_{k+1}^-,q_{k+1}), 0 , 0, f_d^+(q_{k+1}^-, q_{k+1})).
\end{align*}
Then, we can write conditions \eqref{ForceConds} as
\begin{subequations}
\label{forcedDiracStructs}
\begin{align}
(X_d^k, \mathfrak{D}^+L_d(q_k, q_k^+)-\tilde{F}_{d+}(q_k,q_k^+))\in D_{\Delta Q}^{d+}\\
(X_d^k, \mathfrak{D}^-L_d(q_{k+1}^-, q_{k+1})-\tilde{F}_{d-}(q_{k+1}^-,q_{k+1}))\in D_{\Delta Q}^{d-}.
\end{align}
\end{subequations}
mimicking the continuous expression \eqref{ctsForcedDiracStruct}. In \eqref{forcedDiracStructs} we regard $\mathfrak{D}^\pm L_d - \tilde{F}_{d\pm}$ as an abstract expression and impose base-point matching after the subtraction.

\subsection{Numerical Example}
We close with a simple implementation of the (+) forced discrete Dirac equations. We simulate a basic RLC resonator with a single resistor, inductor, and capacitor. We assume some existing voltage in the system but no replenishing voltage source.
\[
\begin{circuitikz}
      \draw (0,0)
      to(0,2)
      to[L=$L$] (4,2)
      to[R=$R$] (4,0)
      to[C=$C$] (0,0);
   \end{circuitikz}
\]
The system is described as a Lagrange--Dirac system in terms of the charges and currents in each component, $(q^C, q^L, q^R, i^C, i^L, i^R)$. Let $L, R,$ and $C$ denote the inductance, resistance, and capacitance of the components. Then the system has Lagrangian function
\begin{equation*}
L(q,i) = \frac{L}{2}(i^L)^2 - \frac{(q^C)^2}{2C}
\end{equation*}
and force
\begin{equation*}
f(q,i) = -i^R R.
\end{equation*}
Kirchhoff's current law imposes the following constraints,
\begin{align*}
i^L-i^R = 0,\\
i^R-i^C = 0.
\end{align*}
Thus, our constraints can be described in terms of the annihilator distribution $\Delta_Q^\circ = span\{\omega^1,\omega^2\}$ for $\omega^1 = dq^L - dq^R$ and $\omega^2 = dq^R-dq^C$. We follow the discretization in \cite{LeOh2011} to construct $\omega_{d+}^{1,2}$ as $\omega_{d+}^a(q_k,q_{k+1}) = \omega^a(q_k, \mathcal{R}_{q_k}^{-1}(q_{k+1}))$ for $\mathcal{R}_q(v) = q+vh$.

Existence of exact discrete quantities in the Dirac setting has not yet been established, so we cannot turn to them to guide the discretization for $L_d$ and $f_d^\pm$. Establishing such quantities is obviously desirable and a topic for future work. For now, we choose a simple $L_d$ and $f_d^\pm$ shown to work well in both the forced, non-Dirac setting \cite{MaWe2001} as well as the nonholonomic integrator setting \cite{CoMa2001}. Namely, we choose $L_d$ and $f_d^\pm$ to be
\begin{equation*}
L_d^{1/2}(q_k,q_k^+) = hL\left(\frac{q_k^+ + q_k}{2}, \frac{q_k^+ -q_k}{h}\right),
\end{equation*}
and
\begin{equation*}
f_d^{1/2\pm} = \frac{h}{2}f\left(\frac{q_k^+ + q_k}{2}, \frac{q_k^+ -q_k}{h}\right).
\end{equation*}
\cite{MaWe2001} introduces these forces as the natural complement to $L_d^{1/2}$. The discretizations $L_d^{1/2}$ and $f_d^{1/2\pm}$ both correspond to choosing midpoint quadrature and a linear boundary-value solution to approximate $L_d^E$ and $f_d^{E\pm}$ in the forced, non-Dirac setting. Thus, the set $(L_d^{1/2},f_d^{1/2\pm})$ is natural and equivalence-preserving by our earlier analysis as well.

Figure \ref{ForcedDiracFig} shows that the simulated charge in the capacitor over time (dotted line) is almost identical to the exact solution (solid line) using a time-step of $h = 0.05$ for 1000 time-steps.

\begin{figure}
\includegraphics[scale=0.6]{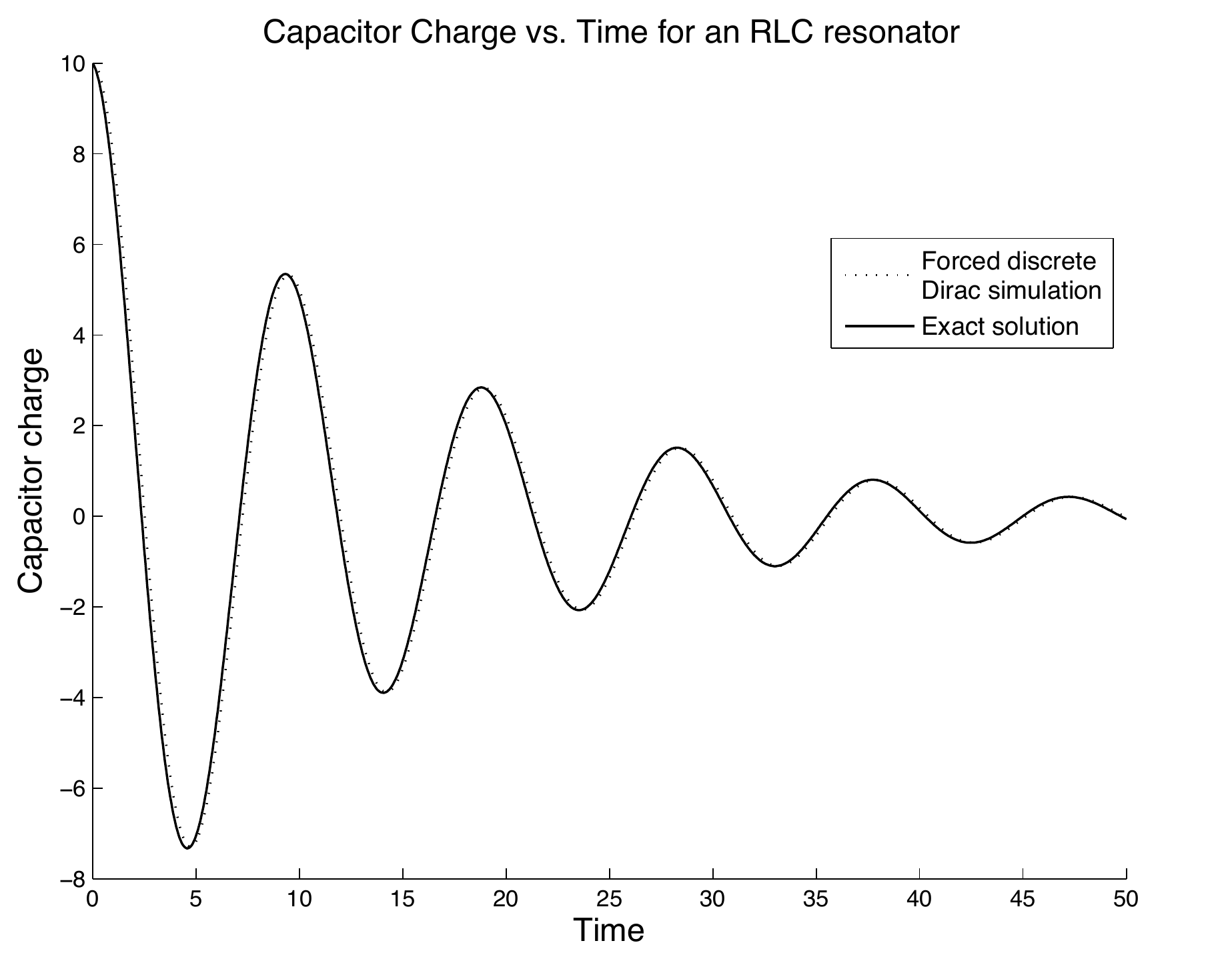}
\caption[A forced discrete Dirac integrator accurately simulates the decaying charge oscillations of a capacitor in an RLC resonator.]{A forced discrete Dirac integrator accurately simulates the decaying charge oscillations of a capacitor in an RLC resonator. The resonator consists of a single loop with one inductor ($L=0.75$), one capacitor ($C=3$), and one resistor ($R=0.1$). The step size is $h=0.05$.}
\label{ForcedDiracFig}
\end{figure}

\section{Conclusions and future work}

In this work we have successfully integrated the presence of external forcing into the discrete Dirac framework originally presented in \cite{LeOh2011} so that our forced discrete Dirac integrators recover the forced DEL integrators of \cite{MaWe2001}. We have also studied the effects of shifting forced representations on the accuracy of numerical simulations and presented a straightforward method for constructing integrators that avoid these spurious effects.

The work is motivated by the applications to the interconnection of discrete Dirac mechanics. Continuous interconnections can be equivalently represented as either constraints or interaction forces acting between subsystems. One would hope to obtain a similar equivalence at the discrete level. In recent work~\cite{PaLe2017}, we have derived discrete interconnections from the point of view of constraints. However, we do not yet have an equivalent representation in terms of discrete interaction forces. There is a lack of equivalence between forced and constrained representations throughout the discrete literature, though recent work has rectified this issue on vectors spaces through Hamel's formalism~\cite{BaZe2015}.

\section{Acknowledgements}

HP has been supported by the NSF Graduate Research Fellowship grant number DGE-1144086. ML has been supported in part by NSF under grants DMS-1010687, CMMI-1029445, DMS-1065972, CMMI-1334759, DMS-1411792, DMS-1345013.

\bibliographystyle{plainnat}
\bibliography{ThesisBib}

\begin{thebibliography}{10}
\providecommand{\natexlab}[1]{#1}
\providecommand{\url}[1]{\texttt{#1}}
\expandafter\ifx\csname urlstyle\endcsname\relax
  \providecommand{\doi}[1]{doi: #1}\else
  \providecommand{\doi}{doi: \begingroup \urlstyle{rm}\Url}\fi

\bibitem[Ball and Zenkov(2015)]{BaZe2015}
{K. R.} Ball and {D. V.} Zenkov.
\newblock Hamel's formalism and variational integrators.
\newblock In {D. E.} Chang, {D. D.} Holm, G.~Patrick, and T.~Ratiu, editors,
  \emph{Geometry, Mechanics, and Dynamics: The Legacy of Jerry Marsden},
  volume~73 of \emph{Fields Institute Communications}, pages 477--506.
  Springer, 2015.

\bibitem[Cort{\'e}s and Mart{\'{\i}}nez(2001)]{CoMa2001}
J.~Cort{\'e}s and S.~Mart{\'{\i}}nez.
\newblock Non-holonomic integrators.
\newblock \emph{Nonlinearity}, 14\penalty0 (5):\penalty0 1365--1392, 2001.

\bibitem[Hall and Leok(2014)]{HaLe2012}
J.~Hall and M.~Leok.
\newblock Spectral variational integrators.
\newblock \emph{Numer. Math.}, 2014.
\newblock \doi{10.1007/s00211-014-0679-0}.

\bibitem[Jacobs and Yoshimura(2014)]{JaYo2014}
H.O. Jacobs and H.~Yoshimura.
\newblock Tensor products of {D}irac structures and interconection in
  {L}agrangian mechanics.
\newblock \emph{Journal of Geometric Mechanics}, 6\penalty0 (1):\penalty0
  67--98, March 2014.

\bibitem[Leok and Ohsawa(2011)]{LeOh2011}
M.~Leok and T.~Ohsawa.
\newblock Variational and geometric structures of discrete dirac mechanics.
\newblock \emph{Foundations of computational mathematics}, 11\penalty0
  (5):\penalty0 529--562, October 2011.

\bibitem[Leok and Shingel(2012)]{LeSh2011b}
M.~Leok and T.~Shingel.
\newblock General techniques for constructing variational integrators.
\newblock \emph{Front. Math. China}, 7\penalty0 (2):\penalty0 273--303, 2012.
\newblock (Special issue on computational mathematics, invited paper).

\bibitem[Marsden and West(2001)]{MaWe2001}
J.E. Marsden and M.~West.
\newblock Discrete mechanics and variational integrators.
\newblock \emph{Acta Numerica}, 10:\penalty0 357--514, May 2001.

\bibitem[Parks and Leok(2017)]{PaLe2017}
H.~Parks and M.~Leok.
\newblock Variational integrators for interconnected {L}agrange--{D}irac
  systems.
\newblock \emph{Journal of Nonlinear Science}, 2017.
\newblock Online First.

\bibitem[Yoshimura and Marsden(2006{\natexlab{a}})]{YoMa2006a}
H.~Yoshimura and J.E. Marsden.
\newblock Dirac structures in {L}agrangian mechanics \emph{Part I: Implicit
  Lagrangian systems}.
\newblock \emph{Journal of Geometry and Physics}, 57\penalty0 (1):\penalty0
  133--156, 2006{\natexlab{a}}.

\bibitem[Yoshimura and Marsden(2006{\natexlab{b}})]{YoMa2006b}
H.~Yoshimura and J.E. Marsden.
\newblock Dirac structures in {L}agrangian mechanics \emph{Part II: Variational
  Structures}.
\newblock \emph{Journal of Geometry and Physics}, 57\penalty0 (1):\penalty0
  209--250, 2006{\natexlab{b}}.

\end{thebibliography}

\end{document}